\documentclass[12pt]{amsart}
\usepackage{amscd,verbatim}
\usepackage{amssymb}
\usepackage[all]{xy}
\usepackage[colorlinks,linkcolor=blue,citecolor=blue,urlcolor=red]{hyperref}

\newcommand{\Spec}{\operatorname{Spec}}

\newcommand{\Rep}{\operatorname{Rep}}
\renewcommand{\Vec}{\operatorname{\bf Vec}}
\newcommand{\Ker}{\operatorname{Ker}}
\newcommand{\IM}{\operatorname{Im}}
\newcommand{\car}{\operatorname{char}}
\newcommand{\Coker}{\operatorname{Coker}}

\newcommand{\Hom}{\operatorname{Hom}}

\newcommand{\End}{\operatorname{End}}
\newcommand{\Br}{\operatorname{Br}}
\newcommand{\LN}{\operatorname{LN}}
\renewcommand{\ln}{\operatorname{ln}}
\newcommand{\Tr}{\operatorname{Tr}}
\newcommand{\Alb}{\operatorname{Alb}}
\newcommand{\Div}{\operatorname{Div}}
\newcommand{\Pic}{\operatorname{Pic}}
\newcommand{\NS}{\operatorname{NS}}

\newcommand{\tr}{{\operatorname{tr}}}
\newcommand{\crys}{{\operatorname{crys}}}
\newcommand{\tors}{{\operatorname{tors}}}

\newcommand{\rg}{{\operatorname{rk}}}
\newcommand{\ord}{{\operatorname{ord}}}
\newcommand{\alg}{{\operatorname{alg}}}

\newcommand{\Fcriso}{\operatorname{\bf Fcriso}}

\newcommand{\by}[1]{\overset{#1}{\longrightarrow}}
\newcommand{\iso}{\by{\sim}}
\newcommand{\osi}{\overset{\sim}{\longleftarrow}}
\newcommand{\inj}{\hookrightarrow}
\newcommand{\tto}{\dashrightarrow}
\newcommand{\surj}{\rightarrow\!\!\!\!\!\rightarrow}
\newcommand{\colim}{\varinjlim}
\renewcommand{\lim}{\varprojlim}

\DeclareFontFamily{U}{wncy}{}
\DeclareFontShape{U}{wncy}{m}{n}{%
<5>wncyr5%
<6>wncyr6%
<7>wncyr7%
<8>wncyr8%
<9>wncyr9%
<10>wncyr10%
<11>wncyr10%
<12>wncyr6%
<14>wncyr7%
<17>wncyr8%
<20>wncyr10%
<25>wncyr10}{}
\DeclareMathAlphabet{\cyr}{U}{wncy}{m}{n}

\newcommand{\sA}{\mathcal{A}}

\newcommand{\sC}{\mathcal{C}}
\newcommand{\sF}{\mathcal{F}}

\newcommand{\sJ}{\mathcal{J}}

\newcommand{\sM}{\mathcal{M}}

\newcommand{\sZ}{\mathcal{Z}}

\newcommand{\Sha}{\cyr{X}}
\newcommand{\sha}{\cyr{x}}

\newcommand{\C}{\mathbf{C}}
\newcommand{\F}{\mathbf{F}}
\newcommand{\G}{\mathbb{G}}
\newcommand{\K}{\mathbf{K}}
\renewcommand{\L}{\mathbb{L}}
\newcommand{\N}{\mathbf{N}}

\newcommand{\Q}{\mathbf{Q}}
\newcommand{\Z}{\mathbf{Z}}

\newcommand{\bG}{\mathbb{G}}

\newcommand{\et}{{\operatorname{\acute{e}t}}}

\newcommand{\ff}{\mathfrak{f}}

\newcommand{\un}{\mathbf{1}}

\renewcommand{\epsilon}{\varepsilon}
\renewcommand{\phi}{\varphi}

\newtheorem{lemme}{Lemma}[section]
\newtheorem{prop}[lemme]{Proposition}
\newtheorem{thm}[lemme]{Theorem}
\newtheorem{Th}{Theorem}
\newtheorem{Thm}{Theorem}

\newtheorem{Corr}[Th]{Corollary}
\theoremstyle{definition}
\newtheorem{defn}[lemme]{Definition}
\theoremstyle{remark}
\newtheorem{rque}[lemme]{Remark}
\newtheorem{rques}[lemme]{Remarks}

\newtheorem{qn}[lemme]{Question}

\newenvironment{thlist}{\begin{list}{\rm{(\roman{enumi})}}%
{\usecounter{enumi}}}%
{\end{list}}

\numberwithin{equation}{section}

\begin{document}

\title[$L$-function of an abelian variety over a function field]{A motivic  formula for the $L$-function of an abelian variety over a function field}
\author{Bruno Kahn}
\address{IMJ-PRG\\Case 247\\
4 place Jussieu\\
75252 Paris Cedex 05\\France}
\email{bruno.kahn@imj-prg.fr}
\date{March 2, 2016}
\begin{abstract}
Let $A$ be an abelian variety over the function field of a smooth projective curve $C$ over an algebraically closed field $k$.  We compute the $l$-adic cohomology groups
\[H^i(C,j_*H^1(\bar A,\Q_l)), \quad j:\eta\inj C\]
in terms of arithmetico-geometric invariants of $A$. We apply this, when $k$ is the algebraic closure of a finite field, to a motivic computation of the $L$-function of $A$.
\end{abstract}
\subjclass[2010]{11G40, 14C15}
\maketitle

\tableofcontents

\section{Introduction}\label{s:coh}

\subsection{The $L$-function of an abelian variety} Let $K$ be a global field, and let $A$ be an abelian variety over $K$. Its $L$-function is classically defined as
\begin{equation}\label{eqL0}
L(A,s) = \prod_{v\in \Sigma_K^f} \det(1-\pi_vN(v)^{-s}\mid H^1_{l_v}(A)^{I_v})^{-1}
\end{equation}
where $\Sigma_K^f$ is the set of non-archimedean places of $K$ and, given $v\in \Sigma_K^f$, $H^1_{l_v}(A)=H^1_\et(A\otimes_K K_s,\Q_{l_v})$ for a separable closure $K_s$ of $K$ and some prime $l_v$ different from the residue characteristic
, $I_v$ is the absolute inertia group and $\pi_v$ is the geometric Frobenius, well-defined modulo $I_v$ as a conjugacy class. This function does not depend on the choice of the $l_v$'s, as a consequence of Weil's proof of the Riemann hypothesis for curves and of  the ``weight-monodromy conjecture", which is known in this case by \cite[exp. IX, Th. 4.3 and Cor. 4.4]{SGA7}. 

In positive characteristic $p$, we may choose $l_v=l$ for a fixed prime $l\ne p$. We have the more precise theorem, due to Grothendieck (\cite[(13), p. 194]{corr-gs}, see also \cite[\S 10]{deligne-constantes}) and Deligne \cite[Th. 3.2.3]{weilII}:

\begin{Thm}\label{t0.1}Suppose $\car K>0$; let $k=\F_q$ be the field of constants of $K$. Then
one has the formula
\[L(A,s) =\frac{P_1(q^{-s})}{P_0(q^{-s})P_2(q^{-s})}\]
where $P_i\in \Z[t]$ with $P_i(0)=1$. The inverse roots of $P_i$ are Weil $q$-numbers of weight $i+1$.
 Moreover, $L(A,s)$ has a functional equation of the form
\[L(A,2-s) = ab^s L(A,s)\]
for suitable integers $a,b$.
\end{Thm}

In this article, we give a formula for the polynomials $P_i$ in terms of \emph{pure motives over $k$}. To express the result, let us take some notation:
\begin{itemize}
\item $B=\Tr_{K/k} A$ is the $K/k$-trace of $A$ (see \cite[App. A]{picfini} or \cite{conrad}); recall that this is (essentially) the largest abelian subvariety of $A$ which is defined over $k$.
\item $\LN(A,K\bar k/\bar k)=A(K\bar k)/B(\bar k)$ is the \emph{geometric Lang-N\'eron group} of $A$, where $\bar k$ is an algebraic closure of $k$: it is finitely generated by the Lang-N\'eron theorem (e.g. \cite[App. B]{picfini}, \cite{conrad} or \cite{ln}). We view it as a Galois representation over $k$.
\item $T_l(E)=\Hom(\Q_l/\Z_l,E)$; $V_l(E)=T_l(E)\otimes \Q$ for an abelian group $E$.  
\end{itemize}


\begin{Th}\label{t0.2}  Let $\sM$ be the category of pure motives over $k$ with rational coefficients, modulo rational equivalence.\footnote{Throughout this paper we use the contravariant convention for pure motives,  as in \cite{kleiman} or \cite{scholl}.} \\
a) We have
\[P_0(t) = Z(h^1(B),t),\quad P_2(t)=Z(h^1(B),qt)\]
where $h^1(B)\in \sM$ is the degree $1$ part of the Chow-K\"unneth decomposition of the motive of $B$ \cite{den-murre}.\\
b) We have
\[P_1(t) = Z(\ln(A,K/k),qt)^{-1}\cdot Z(\sha(A,K/k),t)^{-1}\]
(a product of two polynomials), where $\ln(A,K/k)$ is the Artin motive associated to $\LN(A,K\bar k/\bar k)$ and $\sha(A,K/k)\in \sM$ is  an effective Chow motive  of weight $2$ whose $l$-adic realization is $V_l(\Sha(A,K\bar k))(-1)$, where $\Sha(A,K\bar k)$ is the geometric Tate-\v Safarevi\v c group of $A$.
\end{Th}

In Theorem \ref{t0.2}, we used the \emph{Z-function} of a motive $M\in \sM$ \cite[p. 81]{kleiman}. It is known to be a rational function of $t$, with a functional equation; more precisely, if $M$ is homogeneous of weight $w$, then $Z(M,t)$ is a polynomial or the inverse of a polynomial according as $w$ is odd or even. That its inverse roots are Weil $q$-numbers of weight $w$ depends on \cite{weilI} rather than \cite{weilII}, see Proposition \ref{p3.2} and Remark \ref{r3.1}.  Theorem \ref{t0.2} also provides a proof that $L(A,s)$ is independent of $l$ avoiding \cite[Exp. IX]{SGA7}. Note that  $B$ already appears in \cite[(4.4)]{tate}.

The motive $\sha(A,K/k)$ is really the new character in this story. We construct it ``by hand'' in Proposition \ref{p.sha}; however, we show in \cite[Cor. 8.4]{adjoints} that it is actually canonical and functorial in $A$ (for homomorphisms of abelian varieties).

Theorem \ref{t0.2} ``reduces'' the Birch and Swinnerton-Dyer conjecture for $A$ to the non-vanishing of $Z(\sha(A,K/k),t)$ at $t=q^{-1}$. The existence of $\sha(A,K/k)$ actually yields a simple proof of the following corollary by basically quoting the relevant literature \cite{DRW,milne,IR}:

\begin{Corr}\label{ckt} a) (cf. Schneider \cite[Lemma 2 (i)]{schneider} for $l\ne p$) One has the equality and inequality, for any prime number $l$ ($l=p$ is allowed): 
\begin{multline}\label{eq.corg}
\ord_{s=1} L(A,s)=\rg A(K)+\dim V_l(\Sha(A,K\bar k))^{(1)}\\
\ge \rg A(K)+\dim V_l(\Sha(A,K))
\end{multline}
where $V_l(\Sha(A,K\bar k))^{(1)}$ is the generalised eigenspace for the eigenvalue $1$ of the action of the arithmetic Frobenius on $V_l(\Sha(A,K\bar k))$. (In particular, $\dim  V_l(\Sha(A,K\bar k))^{(1)}$ is independent of $l$.)\\
b) (Kato-Trihan, \cite{kato-trihan}) The following conditions are equivalent:
\begin{thlist}
\item $\ord_{s=1} L(A,s)=\rg A(K)$.
\item $\Sha(A,K)\{l\}$ is finite for some prime $l$.
\item $\Sha(A,K)\{l\}$ is finite for all primes $l$.
\item $\Sha(A,K)$ is finite.
\end{thlist}
\end{Corr}

On the other hand, computing the special value at $s=1$ factor by factor looks very uninspiring, see \S \ref{sv}. It seems that a new idea is needed to nicely relate Theorem \ref{t0.2} to the formula in the Birch and Swinnerton-Dyer conjecture, for example by revisiting Schneider's height pairing \cite[p. 507]{schneider} as a determinant pairing on total derived complexes. I hope to come back to this in a further work.
\bigskip

\enlargethispage*{20pt}

To prove Theorem \ref{t0.2}, we reduce to the case where $A$ is the Jacobian $J$ of a curve $\Gamma$. In this case,  we also get a precise relationship between $L(J,s)$ and the zeta function of a smooth projective $k$-surface spreading $\Gamma$, which was my original motivation for this work. More precisely,   
let  $\Gamma$ be a smooth, projective, 
geometrically irreducible curve over $K$, $C$ the smooth projective $k$-curve with function field $K$ and $S$  a smooth projective  surface over $k$, fibred 
over $C$ by a flat morphism $f$ with generic  fibre $\Gamma$ (its existence is justified in footnote \ref{fn2} below):
\begin{equation}\begin{CD}\label{eq12.0}
\Gamma@>>> S\\
@V{f'}VV @V{f}VV\\
\Spec K @>j>> C\\
&& @V{p}VV\\
&& \Spec k.
\end{CD}\end{equation}

Define (cf. \cite{dpp})
\begin{align*}
L(h^i(\Gamma),s) &=  \prod_{v\in \Sigma_K} \det(1-\pi_vN(v)^{-s}\mid H^i_l(\Gamma)^{I_v})^{(-1)^{i+1}}\\
L(h(\Gamma),s) &= \prod_{i=0}^2 L(h^i(\Gamma),s)
\end{align*}
so that
\begin{gather*}
L(h^0(\Gamma),s)=\zeta(C,s),\quad L(h^2(\Gamma),s)=\zeta(C,s-1),\\
L(h^1(\Gamma),s) = L(J,s)^{-1}
\end{gather*}
(beware the exponent change!).

For convenience, we express the next theorem in terms of the \emph{zeta function of a motive $M\in \sM$}, rather than its $Z$-function:
\[\zeta(M,s)=Z(M,q^{-s}); \quad \zeta(h(X),s)=\zeta(X,s) \text{ for $X$ smooth projective.}\]

\begin{Th}\label{c4.1} a) We have
\[\frac{\zeta(S,s)}{L(K,h(\Gamma),s)}=  \zeta(a(D),s-1)\]
where $a(D)$ is the Artin motive associated to the ``divisor of multiple fibres''
\[ D=\bigoplus_{c\in C_{(0)}} \bar D_c,\quad \bar D_c=\Coker(\Z\by{f^*} \bigoplus_{x\in \text{Supp}(f^{-1}(c))^{(0)}}\Z).\]
b) The function $\displaystyle\frac{L(K,h(\Gamma),s)}{\zeta(\ln(J,K/k),s-1)}$ only depends on the function field $K(\Gamma)$.
\end{Th}

(Here as in the rest of this paper, we write $Z^{(p)}, Z_{(p)}$ for the set of points of codimension/dimension $p$ of a scheme $Z$.)

Theorem \ref{t0.2}, Corollary \ref{ckt} and Theorem \ref{c4.1}  are results on abelian varieties over a global field of positive characteristic. More intriguing to me is that Theorem \ref{t0.2} leads to a definition of the $L$-function of an abelian variety over a \emph{finitely generated field of Kronecker dimension $2$}, and to an analogue of Theorem \ref{c4.1} for the total $L$-function of a surface $S$ over a number field $k$ sitting in a fibration \eqref{eq12.0} (Theorem \ref{p6.1}). This might be viewed a step towards answering the awkwardness of \cite[\S 4]{tatepoles}: meanwhile, it raises more questions than it answers. See \S \ref{s:global} for more details.

\subsection{The method}\label{s1.2} To prove Theorem \ref{t0.2}, we start from Grothendieck's formula for $P_i(t)$:
\begin{equation}\label{eq:groth}
P_i(t) =\det(1-\pi_kt\mid H^i(C\otimes_k \bar k,j_*H^1_l(A)))
\end{equation}
where $\pi_k$ is the geometric Frobenius of $k$, $C$ is the smooth projective $k$-curve with function field $K$ and $j:\Spec K\inj C$ is the inclusion of the generic point. The issue is then to give an expression of the cohomology groups $H^i(C\otimes_k \bar k,j_*H^1_l(A)))$.
 
This follows from a cohomological computation valid in greater generality.  Consider the situation of \eqref{eq12.0}, where $k$ is now any field.\footnote{\label{fn2} Given $f'$  and assuming $k$ perfect, to get such a diagram we can start from a projective model $S_0$ of $\Gamma$ over $\Spec k$ obtained from some very ample divisor on $\Gamma$, then resolve the singularities of the closure in $S_0\times_k C$ of the graph of the rational map $S_0\tto C$ \cite{abyankhar}; the resulting $f$ is automatically flat by \cite[Ch. II, Prop. 9.7]{hartshorne}. When $k$ is imperfect, one may have to pass to a finite purely inseparable extension to get \eqref{eq12.0}, but this is a harmless restriction, compare proof of Proposition \ref{p.sha}.} For a prime number $l$ invertible in $k$, we write 
\[H^2_\tr(\bar S,\Q_l(1))=\Coker\left(\NS(\bar S)\otimes \Q_l\to 
H^2(\bar S,\Q_l(1))\right)\]
where $\NS(\bar S)$ is the N\'eron-Severi group of $\bar S:=S\otimes_k \bar k$, $\bar k$ being an algebraic closure of $k$ as before. Here are two other descriptions of this group:
\begin{equation}\label{eq:kummer}
H^2_\tr(\bar S,\Q_l(1))\simeq V_l(\Br(\bar S))\simeq V_l(\Sha(J,K\bar k)) 
\end{equation}
where $\Br(\bar S)$ is the Brauer group of $\bar S$, $J$ is the Jacobian variety of $\Gamma$ and $\Sha(J,K\bar k)$ denotes its geometric Tate-\v Safarevi\v c group: the first isomorphism follows from the Kummer exact sequence and \cite[Cor. 2.2]{BrII}, and the second one from \cite[pp. 120/121]{BrIII}.  If $k_s$ (resp. $k_p$) is the separable (resp. perfect) closure of $k$ in $\bar k$, then $G_k:=Gal(k_s/k)\iso Gal(\bar k/k_p)$ acts naturally on the groups above. 

\begin{Th}\label{t12.1}There are $G_k$-equivariant isomorphisms
\begin{align*}
H^0(\bar C,j_*R^1f'_*\Q_l(1))&\simeq H^1(\Tr_{K/k} J\otimes_k \bar k,\Q_l(1))\\
H^2(\bar C,j_*R^1f'_*\Q_l(1))&\simeq H^1(\Tr_{K/k} J\otimes_k \bar k,\Q_l)
\end{align*}
(where $\bar C= C\otimes_k \bar k$), and a $G_k$-equivariant exact sequence
\[0\to \LN(J,K\bar k/\bar k)\otimes\Q_l\to H^1(\bar C,j_*R^1f'_*\Q_l(1))\to H^2_\tr(\bar S,\Q_l(1))\to 0.\]
\end{Th}

\begin{Corr}\label{c1.1} Let  $f':A\to\Spec K$ be an abelian variety, where $K=k(C)$ as in \eqref{eq12.0}. There are $G_k$-equivariant isomorphisms
\begin{gather*}
H^0(\bar C,j_*R^1f'_*\Q_l(1))\simeq  H^1(\Tr_{K/k} A\otimes_k \bar k,\Q_l(1))\\
H^2(\bar C,j_*R^1f'_*\Q_l(1))\simeq  H^1(\Tr_{K/k} A\otimes_k \bar k,\Q_l)
\end{gather*}
(where $\bar C:= C\otimes_k \bar k$), and a $G_k$-equivariant exact sequence
\[0\to \LN(A,K\bar k/\bar k)\otimes\Q_l\to H^1(\bar C,j_*R^1f'_*\Q_l(1))\to V_l(\Sha(A,K\bar k))\to 0.\]
\end{Corr}

\begin{rques} 1) It can be shown by a polarisation argument that the exact sequences of Theorem \ref{t12.1} and Corollary \ref{c1.1} are $G_k$-split.\\
2) If $k=\C$, it seems likely that similar results can be obtained for the analytic cohomology of $R^1f'_*\Q(1)$ with techniques similar  to those used in the next section, replacing Kummer sequences by exponential sequences.
\end{rques}

The main technical part of this work is to prove Theorem \ref{t12.1}. The method is to ``$l$-adify'' Grothendieck's computations with $\G_m$ coefficients in \cite[\S 4]{BrIII}\footnote{These computations also appear with less generality in two other expos\'es of the volume \emph{Dix expos\'es sur la cohomologie des sch\'emas}:  \cite[\S 3]{raynaud} and \cite[Th. 3.1]{tate}.}. In a forthcoming work with Am\'\i lcar Pacheco \cite{kp}, we shall extend these results to a general fibration of smooth projective $k$-varieties, with a different and (hopefully) less unpleasant proof. 


\subsection*{Contents of this paper}  Theorem \ref{t12.1}  is proven in Section \ref{s:proof1}; Corollary \ref{c1.1} and Theorem \ref{t0.2} are proven in Section \ref{s:proof2}. In Section \ref{s:comp}, we show how Theorem \ref{t12.1} yields an identity in $K_0$ of a category of $l$-adic representations or pure motives, see Theorem \ref{t12.2}: this identity implies Theorem \ref{c4.1}. In Section \ref{crys}, we recall well-known facts on the crystalline realisation and present them in a convenient way.  In Section \ref{s:finite}, we examine what Theorem \ref{t0.2} teaches us on the functional equation and special values of $L(A,s)$; in particular, we prove Corollary \ref{ckt} in \S \ref{ooz}. Finally, in Section \ref{s:global}, we examine what happens when we replace the finite field $k$ by a global field.

\subsection*{Acknowledgements} This work was partly  inspired by the papers of Hindry-Pacheco \cite{hp} and Hindry-Pacheco-Wazir \cite{hpw}; I would also like to acknowledge several discussions with Am\'\i lcar Pacheco around it. I thank the R\'eseau franco-br\'esilien de math\'ematiques (RFBM) for its support for two visits to Rio de Janeiro in 2008 and 2010.

Theorems \ref{t0.2} (for the Jacobian of a curve),   \ref{c4.1},  \ref{t12.1} and \ref{t12.2} were obtained  in the fall 2008 at the Tata Institute of Fundamental Research of Mumbai during its $p$-adic semester; I thank this institution for its hospitality and R. Sujatha for having invited me, and would like to add a thought for the Mumbai attacks which took place in this period. These results were initially part of a preliminary version of \cite{adjoints}, from which I extracted them. The other results were obtained more recently. 

\section{Proof of Theorem \ref{t12.1}}\label{s:proof1}

In this section, we assume $k=\bar k$ to avoid carrying notation like $C\otimes_k \bar k$, etc. This is harmless because, in general, $\Tr_{K\bar k/\bar k}(J\otimes_k \bar k)=(\Tr_{K/k} J)\otimes_k \bar k$ (\cite[Prop. 6]{picfini} or \cite[Th. 6.8]{conrad}). It is immediate to check that all exact sequences and isomorphisms appearing below are Galois-equivariant. We follow the notation of \cite[\S 4]{BrIII} as much as possible; in particular, $B$ denotes here $j_*\Pic_{\Gamma/K}$ rather than $\Tr_{K/k} J$. We shall use:

\begin{lemme}[\protect{\cite[Cor. 4.3.12]{EGA3}}]\label{l2.0} The fibres of $f:S\to C$ are connected and $f_*\G_m=\G_m$.
\end{lemme}

\subsection{Reduction to the cohomology of the N\'eron  model}\label{redner}

\begin{lemme} \label{eq12.7} Let $\sJ=j_*J$ be the N\'eron  model of $J$ over $C$. There are short exact sequences
\begin{multline*}
0\to (\lim H^{p-1}(C,\sJ)/l^\nu)\otimes \Q\to H^p(C,j_*R^1f'_*\Q_l(1))\\
\to V_l(H^p(C,\sJ))\to 0.
\end{multline*}
\end{lemme}

\begin{proof} Given $c\in C_{(0)}$, write $i_c:c\inj C$ for the corresponding closed immersion and let  $\Phi_c$ be the group of connected  components of the special  fibre of   $\sJ$ at $c$. Then $\Phi_c$ is finite for any $c$ and is $0$ except for a finite number of $c$'s.  Write $\sJ^0=\Ker(\sJ\to \bigoplus_{c\in C_{(0)}} (i_c)_* \Phi_c)$: it is divisible. Since $k$ is algebraically closed, we have isomorphisms $H^p(C,\sJ^0)\iso H^p(C,\sJ)$ for $p>0$ and an injection with finite cokernel  $H^0(C,\sJ^0)\inj H^0(C,\sJ)$. So,
\begin{gather*}(\lim H^*(C,\sJ^0)/l^\nu)\otimes \Q\iso (\lim H^*(C,\sJ)/l^\nu)\otimes \Q,\\
 V_l(H^*(C,\sJ_0))\iso V_l(H^*(C,\sJ)).
 \end{gather*}

To handle the cohomology of $\sJ^0$, we use the Kummer exact sequences
\[0\to{}_{l^\nu} \sJ^0\to \sJ^0\by{l^n}\sJ^0\to 0\]
which yield exact sequences with finite central terms
\[0\to H^{p-1}(C,\sJ^0)/l^\nu\to H^p(C,{}_{l^\nu} \sJ^0)\to {}_{l^\nu}H^p(C, \sJ^0)\to 0\]
hence other exact sequences
\[0\to \lim H^{p-1}(C,\sJ^0)/l^\nu)\otimes \Q\to H^p(C,V_l (\sJ^0))\to V_l(H^p(C, \sJ^0))\to 0.\]

But $V_l(\sJ^0)\iso V_l(\sJ)$; as $R^1f'_*\mu_{l^\nu}\iso {}_{l^\nu} J$ and 
$j_*{}_{l^\nu} J={}_{l^\nu}\sJ$, the lemma follows.
\end{proof}
\enlargethispage*{40pt}

\subsection{Cohomology of $B:=j_*\Pic_{\Gamma/K}$} Define
\[D=\Coker(\Div(C)\by{f^*}\Div(S-\Gamma))\]
(the right hand side may be thought of as ``vertical divisors of $S$ relative to $f$''). We may describe $D$ as follows: for $c\in C_{(0)}$, write  $f^*(c)$ as an effective divisor $\sum_{i\in I} n_i C_i$ on $S$, with the $C_i$'s irreducible and $n_i>0$. Let $D_c=\bigoplus_{i\in I}\Z$ and $\bar D_c=\Coker(\Z\by{f_c}D_c)$, where $f_c(1)=(n_i)_{i\in I}$: thus $\bar D_c=0$ whenever $f$ is smooth over $c$ (Lemma \ref{l2.0}). Then: 
\begin{equation}\label{eqD}
D=\bigoplus_{c\in C_{(0)}} \bar D_c
\end{equation}
(a finite sum).

\begin{lemme}\label{l2.1} There is an isogeny
\[\Pic^0_{S/k}/\Pic^0_{C/k}\to \Tr_{K/k} J \]
and a complex
\[0\to \NS(C)\to \NS(S)\to \Pic(\Gamma)/\Tr_{K/k} J(k)\to 0\]
which, modulo finite groups, is acyclic except at $\NS(S)$, where its homology 
is $D$.
\end{lemme}

\begin{proof} This follows from \cite[Prop. 3.3 and 3.8]{hp} or \cite[3.2 a)]{ln}.
\end{proof}

\begin{lemme}\label{l2.1a} a) There is an exact sequence
\begin{equation}
0\to D\to \Pic(S/C)\to H^0(C,B)\to 0\label{eq12.2}
\end{equation}
where $\Pic(S/C)=H^0(C,\Pic_{S/C})$ and  $B = j_*\Pic_{\Gamma/K}$.\\
b) There is an exact sequence
\begin{gather}
0\to \Pic(C)\to \Pic(S)\to \Pic(S/C)\to 0\label{eq12.1}\\
\intertext{and isomorphisms}
H^n(S,\bG_m)\iso H^{n-1}(C,B) \text{ for } n>1.\label{eq12.4}
\end{gather}
In particular,
\begin{equation}
\Br(S)\iso H^1(C,B)\quad \label{eq12.3}
\end{equation}
and $H^n(C,B)=0$ for $n>3$.
\end{lemme}

\begin{proof} This follows from the computations in \cite[\S 4]{BrIII}. For simplicity, write $P=\Pic_{S/C}$. The homomorphism $\phi:P\to B$ is epi and its kernel is a skyscraper sheaf whose global  sections are $D$: indeed, epi follows from the fact that all residue fields $k(c)$ for $c\in C_{(0)}$ are perfect  (they are all equal to $k$) \cite[p. 114]{BrIII}, and $H^0(C,\Ker \phi)$ is identified to $D$ by the computation of \cite[p. 116, esp. (4.21)]{BrIII}. This yields a), as well as isomorphisms 
\begin{equation}\label{eq2.3}
H^n(C,P)\iso H^n(C,B), \quad  n>0.
\end{equation} 

To prove b), we use the long cohomology exact sequence from \cite[(4.1)]{BrIII} (a consequence of Lemma \ref{l2.0} and \emph{op. cit.}, (3.2)):
\begin{equation}\label{eq2.1}
\dots\to H^n(C,\bG_m)\to H^n(S,\bG_m)\to H^{n-1}(C,P)\to\dots
\end{equation}

For any smooth $k$-variety $Z$ of dimension $d$, one has $H^n(Z,\G_m)=0$ for $n>2d$ by \cite[Exp. X, Cor. 4.3 and 5.2]{SGA4}, the Kummer exact sequences
\[0\to \mu_m\to \G_m\by{m}\G_m\to 0, \quad m \text{ invertible in } k\]
(resp. $0\to \G_m\by{p}\G_m\to \G_m/p\to 0$ if $p=\car k$) and the fact that $H^n(Z,\G_m)$ is torsion for $n>1$ \cite[Prop. 1.4]{BrII}.
Thus $H^n(S,\bG_m)=0$ for $n>4$. Moreover,  $H^2(C,\bG_m)=0$ by \cite[Cor. 1.2]{BrIII} and \cite[Cor. 2.2]{BrII}. b) follows from these facts and \eqref{eq2.3}, \eqref{eq2.1}.
\end{proof}

\subsection{$l$-adic conversion}\label{s.conv} 

\begin{lemme}\label{l2.3} We have natural isomorphisms
\begin{equation}\label{eq2.10}
T_l(\Tr_{K/k} J)\iso  T_l(H^0(C,B))
\end{equation}
\begin{equation}\label{eq2.11}
 (\lim H^0(C,B)/l^\nu)\iso (\Pic(\Gamma)/(\Tr_{K/k} J)(k))\otimes\Z_l.
\end{equation}
\end{lemme}

\begin{proof} The isomorphism $H^0(C,B)=\Pic(\Gamma)$ gives a tautological exact sequence:
\begin{equation}\label{eq12.new}
0\to \Tr_{K/k} J(k)\to H^0(C,B)\to \Pic(\Gamma)/\Tr_{K/k} J(k)\to 0.
\end{equation}

The first term is divisible and, by Lemma \ref{l2.1}, the third term is finitely generated. Hence in the exact sequences derived from \eqref{eq12.new}:
\begin{multline}\label{eq12.new2}
0\to {}_{l^\nu} \Tr_{K/k} J(k)\to {}_{l^\nu} H^0(C,B)\to {}_{l^\nu} \Pic(\Gamma)/\Tr_{K/k} J(k)\\
\to \Tr_{K/k} J(k)/l^\nu\to H^0(C,B)/l^\nu\to (\Pic(\Gamma)/\Tr_{K/k} J(k))/l^\nu\to 0
\end{multline}
we have $\Tr_{K/k} J(k)/l^\nu=0$ and $({}_{l^\nu} \Pic(\Gamma)/\Tr_{K/k} J(k))$, viewed as an inverse system, is Mittag-Leffler  null. Whence the lemma by passing to the inverse limit.
\end{proof}

\subsection{From  $B$ to  $\sJ$} To pass from  $B$ to  $\sJ$,  we work in the abelian category $\sC$ of abelian groups modulo the Serre subcategory of  finite groups, which does not affect the functor $V_l$. 

\begin{lemme}\label{l2.2} In $\sC$, we have
\begin{enumerate}
\item A split exact sequence $0\to H^0(C,\sJ)\to H^0(C,B)\to \Z\to 0$.
\item An isomorphism $\Sha(J,K/k)=H^1(C,\sJ)\simeq \Br(S)$.
\item An isomorphism $H^2(C,\sJ)\{l\}\simeq \IM_{K/k} J\{l\}(-1)$; $H^2(C,\sJ)$ is torsion.
\item $H^p(C,\sJ)\{l\}=0$ for $p\ge 3$.
\end{enumerate}
\end{lemme}

\begin{proof} (1) and (2) are contained in \cite[\S 4]{BrIII}; we include their proof for completeness. (For the first isomorphism of (2), see \cite[(4.44) and (4.45)]{BrIII}.)

In the exact sequence $0\to J\to \Pic_{\Gamma/K}\by{\deg} \Z$, the map $\deg$ is split up to an integer by the choice of a closed point of $\Gamma$. Applying $j_*$, this yields an exact sequence, split in $\sC$
\begin{equation}\label{eq12.6}
0\to \sJ\to B\to\Z\to 0
\end{equation}
which gives (still in $\sC$) split exact sequences
\[0\to H^p(C,\sJ)\to H^p(C,B)\to H^p(C,\Z)\to 0.\]
  
For $p=0$, we get (1). For $p=1$, we get (2) in view of $H^1(C,\Z)=0$ and \eqref{eq12.3}. 
The group $H^p(C,\sJ)$ is torsion for $p>0$ (compare \cite[IX.4.2]{SGA4}), and $cd_l(C)=2$ (loc. cit., 4.6); the structure of $\sJ$ recalled in  \S \ref{redner} then yields (4). 

It remains to prove (3). Using (2) and the fact that $\Br(S)\{l\}$ is a group of cofinite type  \cite[3.2, 3.3]{BrII}, we get surjective maps $H^2(C,{}_{l^\nu}\sJ)\allowbreak \surj {}_{l^\nu} H^2(C,\sJ)$ with bounded kernels. The Weil pairings
\[{}_{l^\nu} J\times {}_{l^\nu}J\to \mu_{l^\nu}\]
and Poincar\'e duality then show that $H^2(C,\sJ)\{l\}$ is Pontrjagin dual in $\sC$ to $H^0(C,T_l(\sJ))\simeq T_l(H^0(C,\sJ))$, and we conclude by (1) and \eqref{eq2.10}.
\end{proof}

\subsection{Conclusion} From Lemmas \ref{l2.3} and \ref{l2.2}, we derive
\[V_l(\Tr_{K/k} J)\iso  V_l(H^0(C,\sJ))\text{ (Lemma \ref{l2.2} (1) and \eqref{eq2.10})}\]
\[V_l(H^1(C,\sJ))\simeq V_l(\Br(S))\simeq H^2_\tr(S,\Q_l(1)) \text{ (Lemma \ref{l2.2} (2) and \eqref{eq:kummer})}\]
\[\LN(J,K/k)\otimes\Q_l\iso (\lim H^0(C,\sJ)/l^\nu)\otimes\Q\text{ (Lemma \ref{l2.2} (1) and \eqref{eq2.11})}\]
\[(\lim H^1(C,\sJ)/l^\nu)\otimes\Q= 0  \text{ (Lemma \ref{l2.2} (2))}\]
\[V_l(H^2(C,\sJ))\simeq V_l(\IM_{K/k} J)(-1)\simeq V_l(\Tr_{K/k} J)(-1) \text{ (Lemma \ref{l2.2} (3))}\] 
\[(\lim H^2(C,\sJ)/l^\nu)\otimes\Q= 0 \text{ (Lemma \ref{l2.2} (3))}\]
\[V_l(H^p(C,\sJ))= (\lim H^p(C,\sJ)/l^\nu)\otimes\Q= 0\text{ for } p>2 \text{ (Lemma \ref{l2.2} (4))}\]
whence finally the isomorphisms and the exact sequence of Theorem 
\ref{t12.1}, using Lemma \ref{eq12.7} and the isomorphisms
\[V_l(A)(-1)\simeq V_l(A)^*\simeq H^1(A,\Q_l)\]
valid for any abelian variety $A$ over an algebraically closed field.\qed

\section{Proofs of Corollary \ref{c1.1} and Theorem \ref{t0.2}}\label{s:proof2}

\enlargethispage*{20pt}

\subsection{The refined Chow-K\"unneth decomposition for a surface} Let $k$ be an arbitrary field, and $S$ a smooth projective $k$-surface. Recall that, by Murre and Scholl \cite[\S 4]{scholl}, there exists a decomposition in $CH^2(S\times S)\otimes \Q$
\begin{equation}\label{eq4.1}
\Delta_S = \sum_{i=0}^4 \pi^i
\end{equation}
where the $\pi^i$'s are orthogonal idempotents lifting the K\"unneth projectors relative to any Weil cohomology theory $H$ such that $H^1(\Alb(S))\allowbreak\by{f^*} H^1(S)$ is an isomorphism for an Albanese map $S\by{f} \Alb(S)$. Recall that any ``classical'' Weil cohomology theory in the sense of \cite[Def. 7.1.4]{kmp} has this property: this follows for $l$-adic cohomology from \cite[Prop. 4.11]{milne2}, for crystalline cohomology from \cite[Rem. II.3.11.2]{DRW} and for the other classical cohomologies in characteristic $0$ from the comparison theorems with $l$-adic cohomology. In \cite[Prop. 7.2.3]{kmp}, \eqref{eq4.1} was refined to
\[\pi^2=\pi^2_\alg+\pi^2_\tr\]
where $\pi^2_\alg$ and $\pi^2_\tr$ are orthogonal projectors, which cut off the ``algebraic'' and ``transcendental'' part of $H^2(S)$, e.g. $\NS(\bar S)\otimes \Q_l(-1)$ and $H^2_\tr(\bar S,\Q_l)$ from $H^2_\et(\bar S,\Q_l)$. 

Write $\sM(k)=\sM$ for the category of Chow motives over $k$ with rational coefficients \cite{scholl}. The above yields a \emph{refined Chow-K\"unneth decomposition} of the motive $h(S)\in \sM$:
\begin{equation}\label{eq4.2}
h(S) = h^0(S)\oplus h^1(S)\oplus h^2_\alg(S)\oplus t^2(S)\oplus h^3(S)\oplus h^4(S).
\end{equation}

Supposing $S$ geometrically connected, we have $h^0(S)\simeq \un$ (unit motive), $h^3(S)\simeq h^1(S)\otimes \L$ ($\L$ the Lefschetz motive), $h^4(S)\simeq \L^2$ and  $h^2_\alg(S)\simeq \NS_S\otimes \L$, where $\NS_S$ is the Artin motive corresponding to $\NS(\bar S)$ viewed as a $G_k$-module.

Recall also that any abelian variety $A$ has a Chow-K\"unneth decomposition by \cite{den-murre}; in particular one has the motive $h^1(A)\in \sM$, a direct summand of $h(A)$.

\subsection{Correspondences at the generic point} If $X$ is a smooth projective variety of dimension $d$ over a field $F$, we write $CH^d_\equiv(X\times_F X)$
for the quotient of the ring of Chow correspondences on $X$ by the ideal generated by those $Z\subset X\times X$ such that $p_1(Z)\ne X$ or $p_2(Z)\ne X$, where $p_1,p_2$ are the two projections $X\times X\to X$ (\emph{cf.} \cite[ex. 16.1.2 (b)]{fulton}.)

\begin{prop}\label{p3.1} a) In the situation of \eqref{eq12.0}, there is a ring isomorphism $CH^1_\equiv(\Gamma\times_K \Gamma)\iso \End_K(J)$, and a ring homomorphism
\[r:CH^1_\equiv(\Gamma\times_K \Gamma)\to CH^2_\equiv(S\times_k S).\]
b) The rings $\End_K(J)\otimes \Q$ and $CH^2_\equiv(S\times_k S)\otimes \Q$ act compatibly on the isomorphisms and the exact sequence of Theorem \ref{t12.1}, as well as on \eqref{eq:kummer}.
\end{prop}

\begin{proof} (See \cite[Th. 9.3 b)]{adjoints} for a more conceptual proof.) 
 a) The first isomorphism is due to Weil \cite[ch. 6, th. 22]{weil}\footnote{The homomorphism is constructed in \cite[ex. 16.1.2 (c)]{fulton}, but its bijectivity is not mentioned\dots}. We have a homomorphism
\[R:Z^1(C\times_K C)\to Z^2(S\times_k S)\]
defined as follows: let $Z\subset \Gamma\times_K \Gamma$ be an irreducible  cycle of codimension $1$. Write $\sZ$ for its closure in $S\times_C S$. We set $R(Z)=$ image of $\sZ$ in $Z^2(S\times_k S)$. One checks that $R$ passes to rational equivalence and to the equivalences $\equiv$, and that the induced map $r$ is compatible with composition of correspondances.

b) This is a long but eventless verification.
\end{proof}


\begin{prop}\label{p.sha}   Let $A$ be an abelian variety over $K$. There exists an effective Chow motive $\sha(A,K/k)\in \sM(k)$ such that $R_l(\sha(A,K/k))=V_l(\Sha(A,K\bar k))(-1)$, for any prime $l\ne \car k$; $\sha(A,K/k)$ is a direct summand of $t^2(S)$ for some smooth projective $k$-surface $S$. Here $R_l:\sM\to \Rep_{\Q_l}^*(G_k)$ denotes the $l$-adic realisation, with values in finite dimensional graded $\Q_l$-vector spaces with continuous $G_k$-action.
\end{prop}

\begin{proof} If $k$ is perfect, write $A$ as a direct summand of a Jacobian $J$, up to isogeny, where $J$ comes from a situation \eqref{eq12.0}. \emph{Via} Proposition \ref{p3.1}, the corresponding  projector $\pi\in \End(J)\otimes\Q$ defines a projector $r(\pi)\in CH^2_\equiv(S\times_k S)\otimes\Q = \End(t^2(S))$ \cite[Th. 7.4.3]{kmp}. Define $\sha(A,K/k)$ as the image of $r(\pi)$. If $k$ is imperfect, we get a situation \eqref{eq12.0} over some finite purely inseparable extension of $k$, and we conclude by using a Frobenius trick as in \cite[Prop. 1.7.2]{birat-pure}.
\end{proof}

\begin{rque} We show in \cite[Cor. 8.4]{adjoints} that the motive $\sha(A,K/k)$ is independent of the choice of $J$, and is \emph{functorial in $A$}.
\end{rque}

\enlargethispage*{30pt}

\subsection{Proofs of Corollary \ref{c1.1} and Theorem \ref{t0.2}} To prove Corollary \ref{c1.1}, write $A$ as a direct summand up to isogeny of the Jacobian $J$ of a  curve $\Gamma$ as in Theorem \ref{t12.1}: $A$ corresponds to a projector $\pi\in \End_K(J)\otimes\Q$. Proposition \ref{p3.1} shows that $\pi$ acts on the isomorphisms and the exact sequence of Theorem \ref{t12.1}. Thus, Corollary \ref{c1.1} is obtained as a ``direct summand" of Theorem \ref{t12.1}.

For completeness, recall:

\begin{prop}\label{p3.2} Let  $M\in \sM=\sM(\F_q)$. Suppose that there is a smooth projective $\F_q$-variety $X$ whose Chow motive $h(X)$ admits a Chow-K\"unneth decomposition relative to $l$-adic cohomology for some $l\nmid q$, and such that $M$ is a direct summand of $h^i(X)$ for some $i\ge 0$. Then $Z(M,t)=P_M(t)^{(-1)^{i+1}}$, where $P_M$ is the inverse characteristic polynomial of the action of the Frobenius endomorphism $\pi_M$ on $H^i_l(M)$. Moreover, $P_M\in \Z[t]$ and its inverse roots are Weil $q$-numbers of weight $i$.
\end{prop}

\begin{proof} The definition from \cite[p. 81]{kleiman} amounts to $Z(M,t)=\break \exp(\sum_{i\ge 1} \tr(\pi_M^n)\frac{t^n}{n})$, where $\tr(\pi^n)\in \Q$ is computed in the rigid $\otimes$-category $\sM$. Then $\tr(\pi_M^n)=\tr(R_l(\pi_M)^n)$. Hence the first statement, since $H^j(M)=0$ for $j\ne i$ by hypothesis. The second one follows from \cite{weilI}, since $P_M$ obviously divides $\det(1-\pi_Xt\mid H^i_l(X))$.
\end{proof}

\begin{rque}\label{r3.1} Using \cite{KM},  one can show that the conclusion of  Proposition \ref{p3.2} holds if one only assumes that $M$ is effective and that $H^j_l(M)\ne 0$ for $j\ne i$ (one may even use crystalline cohomology instead of $l$-adic cohomology). But \cite{KM} relies on \cite{weilII}, hence this result is less elementary than the one in Proposition \ref{p3.2}.
\end{rque}

In view of Theorem \ref{t0.1}, Theorem \ref{t0.2} immediately follows from \eqref{eq:groth}, Corollary \ref{c1.1}, Proposition \ref{p.sha} and Proposition \ref{p3.2}.

\section{Comparing classes in $K_0$; proof of Theorem \ref{c4.1}} \label{s:comp} 

 We come back again to the situation of \eqref{eq12.0}, $k$ being arbitrary. One might want to compare
\[R(pf)_*\Q_l\]
and
\[Rp_*j_*Rf'_*\Q_l\]
in the derived category of $\Rep_{\Q_l}^*(G_k)$ (cf. Proposition \ref{p.sha}). Unfortunately this has no meaning, because $j_*$ 
has no meaning in this derived category.

On the other hand, let $\K_l$ be the Grothendieck group of the abelian category  $\Rep_{\Q_l}^*(G_k)$. We may consider in $\K_l$:

\begin{enumerate}
\item $H_l=[H^*_l(\bar S)]$, the alternating sum of the $l$-adic  cohomology groups  of $S$;
\item $H'_l=[R^*p_*j_*R^*f'_*\Q_l]$ (9 terms). 
\end{enumerate}

We may also consider $D$ 
as a discrete topological $G_k$-module.

\begin{thm}\label{t12.2} $H_l-H'_l=[D\otimes\Q_l(-1)]$.
\end{thm}

\begin{proof} For simplicity, put $B:=\Tr_{K/k} J$. Let $\K$ be the Grothendieck group of the additive category $\sM$, and let 
$R_l:\K\to \K_l$ be the homomorphism given by $l$-adic  realisation. (To avoid any confusion, we adopt cohomological notation as in \cite{kleiman,scholl}, but contrary to \cite{kmp}.) Then $H_l = R_l(h)$, with
$h=h(S)$. Similarly, there exists a canonical $h'\in \K$  such that $H'_l=R_l(h')$. Indeed,  Theorem 
\ref{t12.1} shows that
\begin{gather*}
[p_*j_*R^1f'_*\Q_l] = R_l([h^1(B)]),\quad [R^2p_*j_*R^1f'_*\Q_l] = R_l([h^1(B)\otimes \L]),\\
[R^1p_*j_*R^1f'_*\Q_l] = R_l([t^2(S)]+[\ln(J,K/k)\otimes \L])
\end{gather*}
where $\ln(J,K/k)$ is the Artin motive associated to the Galois module $\LN(J,K\bar k/\bar k)$. (Recall that $R_l(\L)=\Q_l(-1)$.)

Using $f'_*\Q_l = \Q_l$ and $R^2f'_*\Q_l=\Q_l(-1)$, we similarly get
\[
[R^qp_*j_*f'_*\Q_l] = R_l([h^q(C)]),\quad [R^qp_*j_*R^2f'_*\Q_l] = R_l([h^q(C)\otimes \L]).
\]

We may then take
\begin{multline*}h' = \sum_{q=0}^2 (-1)^q[h^q(C)] + \sum_{q=0}^2 (-1)^q[h^q(C)\otimes \L]\\
-\left([h^1(B)]
+[h^1(B)\otimes \L]-([t^2(S)]+[\ln(J,K/k)\otimes \L])\right).
\end{multline*}

To prove Theorem \ref{t12.2}, it therefore suffices to show:
\begin{equation}\label{eq12.8} h-h'= [D\otimes \L].
\end{equation}

From Lemma \ref{l2.1}, one gets identities in $\K$:
\[[h^1(S)]=[h^1(C)]+ [h^1(B)],\quad [h^3(S)]=[h^1(C)\otimes \L]+ [h^1(B)\otimes \L],\]
\begin{equation}\label{eq4.3}
[\NS_S]=[\NS_C] + [\un] + [\ln(J,K/k)] + [D]= 2[\un]+  [\ln(J,K/k)] + [D]
\end{equation}
from which \eqref{eq12.8} easily follows.
\end{proof}

Theorem \ref{c4.1} a) immediately follows from Theorem \ref{t12.2}. For b), we note that all terms appearing in \eqref{eq4.2} are birational invariants, except for $h^2_\alg(S)=\NS_S\otimes \L$. Using \eqref{eq4.3} again, we get from \eqref{eq12.8} the identity
\[h-[\NS_S\otimes \L]=h'- 2[\L]-[\ln(J,K/k)\otimes \L]\] 
so that the right hand side only depends on $k(S)=K(\Gamma)$.

\section{The crystalline realisation}\label{crys}
\enlargethispage*{30pt}

This section prepares the proof of Corollary \ref{ckt}, which will be given in the next section. Here $k$ is any perfect field of characteristic $p>0$.

\subsection{Isocrystals} We rely here on the crystal-clear exposition of Saavedra \cite[Ch. VI, \S 3]{saa}.

Let $W(k)$ be the ring of Witt vectors on $k$ and $K(k)$ be the field of fractions of $W(k)$. The Frobenius automorphism $x\mapsto x^p$ of $k$ lifts to an endomorphism on $W(k)$ and an automorphism of $K(k)$, written $\sigma$: we have $K(k)^\sigma=\Q_p$. A \emph{$k$-isocrystal} is a finite-dimensional $K(k)$-vector space $V$ provided with a $\sigma$-linear automorphism $F_V$. $k$-isocrystals form a $\Q_p$-linear tannakian category $\Fcriso(k)$, provided with a canonical $K(k)$-valued fibre functor (forgetting $F_M$) \cite[VI.3.2.1]{saa}. We have
\[\Fcriso(k)(\un,V)=V^{F_V} = \{v\in V\mid F_V v=v\}\]
for $V\in \Fcriso(k)$, where $\un=(K(k),\sigma)$ is the unit object. For $n\in \Z$, we write more generally
\begin{equation}\label{eq6.4}
V^{(n)} = V^{F_V=p^n} = \Fcriso(k)(\L_\crys^{n},V)=\Fcriso(k)(\un,V(n))
\end{equation}
where $V(n)=V\otimes \L_\crys^{-n}$ with $\L_\crys :=(K(k),p\sigma)$.

\subsection{The realisation} By \cite[VI.4.1.4.3]{saa}, the formal properties of crystalline cohomology yield a $\otimes$-functor
\[R_p:\sM(k)\to \Fcriso(k).\]

This functor sends the motive of a smooth projective variety $X$ to $H^*_\crys(X/W(k))\otimes_{W(k)} K(k)$ and the Lefschetz motive $\L$ to $\L_\crys$.

\subsection{The case of a finite field} Suppose that $k=\F_q$, with $q=p^m$. Then any object $M\in \sM(k)$ has its \emph{Frobenius endomorphism} $\pi_M$: if $M=h(X)$ for a smooth projective variety $X$, $\pi_M =\pi_X$ is the graph of the Frobenius endomorphism $F^m$ on $X$. This implies:

\begin{lemme}\label{l6.1} The action of $\pi_M$ on $R_p(M)$ equals that of $F_{R_p(M)}^m$.\qed
\end{lemme}

Let $\bar k$ be an algebraic closure of $k$. There is an obvious functor
\begin{equation}\label{eq6.5}
\Fcriso(k)\to \Fcriso(\bar k),\quad V\mapsto \bar V:=V\otimes_{K(k)} K(\bar k)
\end{equation}
which is compatible with the extension of scalars $\sM(k)\to \sM(\bar k)$ via the realisation functors $R_p$ for $k$ and $\bar k$. Moreover $F_M^m$ is $K(k)$-linear, therefore one can talk of its eigenvalues. We have the following result of Milne \cite[Lemma 5.1]{milne}:

\begin{lemme}\label{lmilne} One has an equality
\[\det(1 - \gamma t\mid \bar M^{(n)} ) =\prod_{v(a) = v(q^n)} (1 - (q^n/at))\]
where $\gamma$ is the arithmetic Frobenius and $a$ runs through the eigenvalues of $F_M^m$ having same valuation as $q^n$.
\end{lemme}

\subsection{Logarithmic Hodge-Witt cohomology} 

\begin{prop} \label{p6.2} Let $X/k$ be smooth projective. Then, for any $i,n\in \Z$, there is a canonical isomorphism
\[H^i(X,\Q_p(n))\iso (H^i_\crys(X/W(k))\otimes_{W(k)} K(k))^{(n)} \]
where the left hand side is logarithmic Hodge-Witt cohomology as in Milne \cite[p. 309]{milne}.
\end{prop}

\begin{proof} This is \cite[Prop. 1.15]{milne}, but unfortunately its proof is garbled (the last line of loc. cit., p. 310 is wrong). Let us recapitulate it. For simplicity, let  $W=W(k)$ and $K=K(k)$.

1) The slope spectral sequence
\[E_1^{i,j} = H^j(X,W\Omega^i)\Rightarrow H^{i+j}(X,W\Omega^\cdot)\simeq H^{i+j}_\crys(X/W)\]
degenerates up to torsion, yielding canonical isomorphisms of $k$-isocrystals
\[H^{i-n}(X,W\Omega^n)\otimes_W K\iso (H^i(X/W)\otimes_W K)_{[n,n+1[}  \]
where the index $[n,n+1[$ means the sum of summands of slope $\lambda$ for $n\le \lambda<n+1$  \cite[Th. 3.2 p. 615 and  (3.5.4) p. 616]{DRW}.

2) If $k$ is algebraically closed, the homomorphism
\[H^i(X,\Z_p(n)):=H^{i-n}(X,W\Omega^n_{\log})\to H^{i-n}(X,W\Omega^n)^F\]
is bijective \cite[Cor. 3.5 p. 194]{IR}.

3) In general, descend from $\bar k$ to $k$ by taking Galois invariants.
\end{proof}

By \cite[\S 2]{milne} and \cite{gros}, Chow correspondences act on logarithmic Hodge-Witt cohomology by respecting the isomorphisms of Proposition \ref{p6.2}: this yields a functor
\[H^*(-,\Q_p(n)): \sM(k)\to \Vec_{\Q_p}^*\]
and a natural isomorphism
\begin{equation}\label{eq6.2}
H^*(M,\Q_p(n))\overset{\sim}{\to} \Fcriso(k)(\L^n_\crys,R_p(M)),\;  M\in \sM(k).
\end{equation}

\subsection{The Brauer group and the Tate-\v Safarevi\v c group} We have

\begin{prop}[\protect{\cite[(5.8.5) p. 629]{DRW}}] Let $k$ be algebraically closed and $X/k$ be smooth projective. Then there is an exact sequence 
\[
0\to \NS(X)\otimes\Z_p\to H^2(X,\Z_p(1))\to T_p(\Br(X))\to 0.
\]
\end{prop}

As before, Chow correspondences act on this exact sequence. Therefore if $X=S$ is a surface, applying the projector $\pi^2_\tr$ defining $t^2(S)$, we get an isomorphism
\[H^2(t^2(S),\Q_p(1))\simeq V_p(\Br(S))\]
hence, taking \eqref{eq6.2} into account:
\[\Fcriso(k)(\L_\crys,R_p(t^2(S)))\simeq V_p(\Br(S)).\]

If now $K/k$ is a function field in one variable and $A$ is an abelian variety over $K$, using the projector $r(\pi)$ from the proof of Proposition \ref{p.sha}, we get an isomorphism
\begin{equation}\label{eq6.3}
\Fcriso(k)(\L_\crys,R_p(\sha(A,K/k)))\simeq V_p(\Sha(A,K))
\end{equation}
using Lemma \ref{l2.2} (2).

\section{Functional equation, order of zero and special value; proof of Corollary \ref{ckt}}\label{s:finite}

\subsection{Functional equation} Recall the functional equation of the zeta function of a pure  motive $M$ of weight $w$ over $k=\F_q$:
\[\zeta(M^*,-s) = \det(M) (-q^{-s})^{\chi(M)} \zeta(M,s)\]
where $M^*$ is the dual of $M$, $\chi(M)$ is the Euler  characteristic of $M$ (computed for example with the help of its $l$-adic  realisation) and 
\[\det(M)=\pm q^{w\chi(M)/2}\] 
is the determinant of the Frobenius endomorphism of $M$. Applying this to Theorem \ref{t0.2}, we get the  following functional equation for $L(K,A,s)$:
\[L(K,A,2-s) = a (-q^{-s})^\beta{}L(K,A,s)\]
with
\begin{align*}
\beta &= -2\chi(h^1(B))-\chi(\sha(A,K/k))-\chi(\ln(A,K/k))\\
&= 4 \dim B - \dim V_l(\Sha(A,K\bar k)) - \rg A(K\bar k) \quad\text{for any prime } l\\
a &= \left(\det h^1(B)\det h^1(B)(-1)\det \sha(A,K/k)\det\ln(A,K/k)(-1)\right)^{-1}\\
&= \pm q^\beta.
\end{align*}

Note that \cite[(C') p. 193]{corr-gs} yields a different expression for $\beta$:
\[
\beta = \chi(j_*H^1(\bar A,\Q_l))=  2-2g - \deg(\ff)
\]
where $g$ is the genus of $C$ and $\ff$ is the conductor of $A$ (relative to $K/k$), and the second equality follows from \cite[Th. 1]{raynaud}. Is there a direct proof that these two expressions coincide?

\subsection{Order of zero: proof of Corollary \ref{ckt}}\label{ooz} a) It suffices to show that $\ord_{s=1} Z(\sha(A,K/k))= \dim V_l(\Sha(A,K\bar k))^{(1)}$ for any prime $l$. This order can be computed through the action of the Frobenius endomorphism $\pi_\sha$ of $\sha(A,K/k)(1)$ on $R(\sha(A,K/k)(1))=R_l(\sha(A,K/k))(-1)$ for any realisation functor $R$ on $\sM(k)$. If we use the $l$-adic realisation $R_l$ for $l\ne p$, the claim is clear by Proposition \ref{p.sha}  since $R_l(\pi_\sha)$ acts as the inverse of $\gamma$. If we now take $R_p$, we find from \eqref{eq6.4}, \eqref{eq6.5} and  \eqref{eq6.3}:
\[R_p(\sha(A,K\bar k/\bar k)(1))^{(0)}\simeq \overline{ R_p(\sha(A,K/k)(1))}^{(0)}\simeq V_p(\Sha(A,K\bar k)).\]

By Lemma \ref{lmilne}, we have
\[\det(1 - \gamma t\mid \overline{R_p(\sha(A,K/k)(1))}^{(0)} ) =\prod_{v(a) = 0} (1 - 1/at)\]
where $a$ runs through the eigenvalues of $F_\sha^m$ acting on $R_p(\sha(A,K/k)(1))$, with valuation $0$. By lemma \ref{l6.1}, $F_\sha^m=R_p(\pi_\sha)$, so we are done.

Note that this argument does \emph{not} show that $\dim V_l(\Sha(A,K))$ is independent of $l$. 

b) For any $l$, $V_l(\Sha(A,K\bar k))^{(1)}=0$ $\iff$ $V_l(\Sha(A,K\bar k))^{G_k}=V_l(\Sha(A,K))=0$ $\iff$ $\Sha(A,K)\{l\}$ is finite. In view of a), this shows (i) $\iff$ (ii) $\iff$ (iii). Of course (iii) $\Rightarrow$ (iv) is classical, but let us give a proof using the motive $\sha(A,K/k)$. Let $P$ be the characteristic polynomial of $\pi_\sha$ acting on the realisations of $\sha(A,K/k)$: by Propositions \ref{p.sha} and \ref{p3.2}, it is independent of $l$ and (ii) implies that $P(q)\ne 0$. Hence $\gamma$ acts invertibly on $\Sha(A,K\bar k)\{l\}$ for $l\nmid P(q)$ and $\Sha(A,K\bar k)\{l\}^\gamma=0$ for such $l$. It remains to study $\Ker(\Sha(A,K)\to \Sha(A,K\bar k)^\gamma)$: by the Hochschild-Serre spectral sequence for the Galois cohomology of $A$, this kernel is contained in the group $H^1(G_k,A(K\bar k))$. The exact sequence
\[0\to B(\bar k)\to A(K\bar k)\to \LN(A,K\bar k/\bar k)\to 0,\]
Lang's theorem and the cohomological dimension of $G_k$ yield an isomorphism
\[H^1(G_k,A(K\bar k))\iso H^1(G_k,\LN(A,K\bar k/\bar k)).\]

But the right hand side is finite since $\LN(A,K\bar k/\bar k)$ is finitely generated. Hence $\Sha(A,K)\{l\}=0$ for $l$ large enough, concluding the proof of (iv).

\subsection{Special value} \label{sv}
It is less obvious to relate Theorem \ref{t0.2} to the value of the principal  part in the Birch and Swinnerton-Dyer conjecture (\cite[Theorem p. 509]{schneider}, \cite{kato-trihan}):
\begin{equation}\label{eq:bsd}
\operatornamewithlimits{lim}_{s\to 1} \frac{L(A,s)}{(s-1)^{-\rho}}\sim \pm q^\rho\frac{|\Sha(A,K)||\det\langle ,\rangle_{A(K)}|}{|A(K)_\tors||A'(K)_\tors|}\prod_{c\in C}|\Phi_c(k(c))|,
\end{equation}
where $\rho= \rg A(K)$, $\langle,\rangle_{A(K)}$ is the height pairing constructed in \cite[p. 502]{schneider}, $A'$ is the dual abelian variety to $A$ and the $\Phi_c$ are the groups of connected components of the N\'eron model of $A$ over $C$, as in \S \ref{redner}.  

It seems that the explicit expression of $L(A,s)$ could actually be used to provide an expression of the left hand side of \eqref{eq:bsd} independently of the Birch and Swinnerton-Dyer conjecture, in the spirit of \eqref{eq.corg}. This can presumably be done by the method of \cite{schneider}: I hope to come back to it in the future. Let me only note that in Theorem \ref{t0.2}, the factors $Z(h^1(B),q^{-s})$ and $Z(h^1(B),q^{1-s})$ respectively contribute by $q^{-g(B)}|B(k)|$ and $|B(k)|$ (as usual, $B:=\Tr_{K/k} A$), while $Z(\ln(A,K/k),q^{1-s})$ contributes by
\[\pm q^{-\rg A(K)} \frac{\det \langle,\rangle_{A(K\bar k)}}{\det \langle,\rangle_{A(K)}} \frac{|A(K)_\tors/B(k)|}{|(\LN(A,K\bar k/\bar k)_\gamma)_\tors|} \]
where $\langle,\rangle_{A(K)}$ and  $\langle,\rangle_{A(K\bar k)}$ are as above. This folllows from the elementary lemma, in the spirit of \cite[Lemma z.4]{tate}:

\begin{lemme} Let $\langle,\rangle:M\times M'\to \Q$ be a $\Q$-non-degenerate pairing between finitely generated abelian groups. Suppose $M$ and $M'$ are provided with operators $\gamma,\gamma'$ which are adjoint for the pairing, and $\Q$-semi-simple (e.g., $\gamma$ is of finite order). Let $P=\det(1-\gamma T)$ be the inverse characteristic polynomial of $\gamma$ acting on $M_\Q$. Then $\rho:=\ord_{T=1} P=\rg M^\gamma$ and, if $P'=P/(1-T)^\rho$,
\[|P'(1)|=\frac{\det \langle,\rangle}{\det \langle,\rangle^\gamma} \frac{|(M^\gamma)_\tors|}{|(M_\gamma)_\tors|}\]
where $M^\gamma$ (resp. $M_\gamma$) denotes the $\gamma$-invariants (resp. coinvariants) of $\gamma$ and $\langle,\rangle^\gamma$ is the ($\Q$-non-degenerate) pairing induced by $\langle,\rangle$ on $M^\gamma\times M'^{\gamma'}$.
\end{lemme}

\section{Surfaces over a global field}\label{s:global}

Let $K$ be a finitely generated field [over its prime subfield]. In \cite[\S 4]{tatepoles}, Tate associates to a smooth, projective, geometrically irreducible $K$-variety $V$ the collection of zeta functions $\zeta(X,s)$ for smooth projective models $X\to Y$of $V$, where $X,Y$ are of finite type over $\Z$, $Y$ is a regular model of $K$ and $X$ is irreducible. For $y\in Y_{(0)}$, the zeta function of the fibre $X_y$ may be factored as an alternating product of polynomials $P_i(X_y)$ corresponding to the $l$-adic cohomology groups of $X_y$, and the $P_i(X_y)$ are now known to be independent of $l$ by \cite{weilI}. This provides a factorisation
\[\zeta(X,s)=\prod_{i=0}^{2d} \Phi_i(X/Y,s)^{(-1)^{i}}, \quad d=\dim V \]
and, for each $i$, $\Phi_i(X/Y,s)$ converges absolutely for $\Re(s)>\delta+i/2$ where $\delta=\dim Y$ is the \emph{Kronecker dimension of $K$}, again by \cite{weilI}.

Tate then observes that, assuming one knows that $\Phi_i(X/Y,s)$ has an analytic continuation for $\Re(s)>\delta+i/2-1$, the order its of zeroes and poles  in the strip $\delta+ i/2-1<\Re(s)\le \delta+i/2$ depends only on $V$ and not on the choice of $X\to Y$, which allows him to formulate his conjectures on these orders (for $s\in \N$).

The question arises whether one can do better and associate canonical analytic functions to $V$ rather that functions $\Phi_i(X/Y,s)$ depending on a model, for example to make sense of special values. If $K$ is a global field, this is the subject of Serre's paper \cite{dpp}. To the best of my knowledge, this question is open when  $\delta>1$; the purpose of this section is to offer a partial answer when $\delta=2$, in the cases $i=0,1$. 

We may write $K$ as a  function field in one variable over some global field $k$; without loss of generality, we may assume $k$ algebraically closed in $K$. There is a finite purely inseparable extension $k'/k$ such that  $Kk'=k'(C)$ for some smooth projective $k'$-curve $C$; of course $k'=k$ if $\car K=0$, and $k$ is then unique as the algebraic closure of $\Q$ in $K$.  In view of Theorem \ref{c4.1} b), we set:

\begin{defn}\label{d6.2}  a) If $\car K=0$, $L(K,h^0(V),s)=L(k,h(C),s)$.\\
b) If $\car K>0$, $L(K,h^0(V),s)=\displaystyle\frac{L(k',h(C),s)}{\zeta(\ln(J(C),k'/\F_q),s-1)}$. Here $J(C)$ is the Jacobian of $C$ and $\F_q$ is the field of constants of $k'$. 
\end{defn} 

Note that if $Y$ is a regular model of $K$ over $\Z$, then $L(K,h^0(V),s)$ and $\zeta(Y,s)$ only differ by a finite number of Euler factors. 

We now pass to the case $i=1$. In view of Theorem \ref{t0.2}, we set:

\begin{defn}\label{d6.1} a) If $A$ is an abelian variety over $K$, 
\begin{multline*}
L(K,h^1(A),s)= L(k,h^1(\Tr_{K/k} A),s)L(k,h^1(\Tr_{K/k} A),s-1)\\
L(k,\sha(A,K/k),s)L(k,\ln(A,K/k),s-1)
\end{multline*}
where the right hand side is defined in terms of $l$-adic realisations.\\
b) $L(K,h^1(V),s)=L(K,h^1(A),s)$, where $A$ is the Albanese variety of $V$.
\end{defn}


Definition \ref{d6.1} is independent of the choice of $l$ (invertible in $k$) because this is so for each individual factor in a): for $\sha(A,K/k)$ it follows from \cite[Satz 2.13]{rz} and  \cite[Cor. 0.6]{tsaito}, using Proposition \ref{p.sha}. 


\begin{qn} If $\car K >0$, is Definition \ref{d6.1} independent of the choice of $k$? 
\end{qn}

One may have to do a normalization similar to the one in Definition \ref{d6.2} b) to get a positive answer.

\begin{qn} Can one interpret $L(K,h^1(A),s)$, \emph{via} a trace formula, as an ``Euler" product of the form
\[L(C,j_* H^1_l(A),s) = \prod_{x\in C_{(0)}} L(k(x),i_x^*H^1_l(\sA),s)\]
where $\sA$ is the N\'eron model of $A$ over $C$?
\end{qn}

 (It is not even  clear that the right hand side converges!)

Let us now place ourselves again in the situation of \eqref{eq12.0}. It is natural to set $L(K,h^2(\Gamma),s)\allowbreak= L(k,h^0(\Gamma),s-1)$, and $L(K,h(\Gamma),s)= \prod_{i=0}^2L(k,h^i(\Gamma),s)$. Theorem \ref{t12.2} then gives the following analogue to Theorem \ref{c4.1} a):

\begin{thm}\label{p6.1} If $\car K=0$, one has
\[\frac{L(k,h(S),s)}{L(K,h(\Gamma),s)}= L(k,a(D),s-1).\qed\]
\end{thm}

If $\car K>0$, one has to take the normalization of Definition \ref{d6.2} b) into account; since Definition \ref{d6.1} may not be optimal, we skip this.
\enlargethispage*{30pt}

\begin{qn} The height pairing defined by Schneider in \cite[p. 507]{schneider}:
\[\sA^0(\bar C)\times A'(K\bar k)\to \Pic(\bar C)\]
induces a Galois-equivariant pairing
\begin{equation}\label{eq7.1}
\sA^0(\bar C)/B(\bar k)\times \LN(A',K\bar k/\bar k)\to \Z
\end{equation}
because $B(\bar k)$ and $B'(\bar k)$ are divisible; moreover, it presumably restricts to a pairing
\begin{equation}\label{eq6.1}B(\bar k)\times \LN(A',K\bar k/\bar k)\to \Pic^0(\bar C).
\end{equation}

One way to justify \eqref{eq6.1} would be to show that the functor $S\mapsto \Gamma(C\times_k S,\sA'\times_k S)$ on $k$-schemes of finite type is representable by a $k$-group scheme of finite type with connected component $B$, and that Schneider's pairing emanates from a pairing of $k$-group schemes. Then \eqref{eq6.1} would induce a Galois-equivariant homomorphism
\begin{equation}\label{eq7.2}
\LN(A',K\bar k/\bar k)\to \Hom_{\bar k}(B,J).
\end{equation}

Can one use \eqref{eq7.1} and \eqref{eq7.2} to describe the special values of $L(K,h^1(A),s)$?
\end{qn}


\begin{thebibliography}{SGA7}
\bibitem[A]{abyankhar} S. Abyankhar {\it Resolution of singularities of algebraic surfaces}, Algebraic Geometry (Internat. Colloq., Tata Inst. Fund. Res., Bombay, 1968) 1--11, Oxford Univ. Press, 1969.
\bibitem[C]{conrad} B. Conrad {\it Chow's $K/k$-trace and $K/k$-image, and the Lang-N\'eron theorem}, L'Ens. Math. {\bf 52} (2006), 37--108.
\bibitem[Dc]{deligne-constantes} P. Deligne {\it Les constantes des \'equations fonctionnelles des fonctions $L$},  Lect. Notes in Math. {\bf 349}, Springer, 1973, 501--597.
\bibitem[W.I]{weilI} P. Deligne {\it La conjecture de Weil, I}, Publ. Math. IH\'ES {\bf 43} (1974),
273--307.
\bibitem[W.II]{weilII} P. Deligne {\it La conjecture de Weil, II},
Publ. Math. IH\'ES {\bf 52} (1980), 137--252.
\bibitem[DM]{den-murre} C. Deninger, J.-P. Murre {\it Motivic decomposition of abelian schemes and the Fourier transform}, J. reine angew. Math. {\bf 422} (1991), 201--219.
\bibitem[F]{fulton} W. Fulton Intersection theory, Springer, 1984.
\bibitem[Gr]{gros} M. Gros Classes de Chern et classes de cycles en cohomologie de Hodge-Witt
logarithmique, M\'em. SMF {\bf 21} (1985), 1--87.
\bibitem[GS]{corr-gs} A. Grothendieck Lettre du 30.9.1964, {\it in} P. Colmez, J.-P. Serre, Correspondance Grothendieck-Serre, Documents math\'ematiques {\bf 2}, SMF, 2001.
\bibitem[Br.II]{BrII} A. Grothendieck {\it Le groupe de Brauer, II: th\'eorie cohomologique}, 
{\it in} Dix expos\'es sur la cohomologie des sch\'emas, Masson--North Holland, 1970, 67--87.
\bibitem[Br.III]{BrIII} A. Grothendieck {\it Le groupe de Brauer, III: exemples et compl\'ements}, 
{\it in} Dix expos\'es sur la cohomologie des sch\'emas, Masson--North Holland, 1970, 88--188.
\bibitem[H]{hartshorne} R. Hartshorne Algebraic geometry, Springer, 1977.
\bibitem[HP]{hp} M. Hindry, A. Pacheco {\it Sur le rang des jacobiennes sur les corps de fonctions}, 
Bull. SMF {\bf 133} (2005), 275--295.
\bibitem[HPW]{hpw} M. Hindry, A. Pacheco, R. Wazir {\it Fibrations et
conjecture de Tate}, J. Number Theory {\bf 112} (2005), 345--368.
\bibitem[I]{DRW} L. Illusie {\it Complexe de de Rham-Witt et cohomologie cristalline}, Ann. Sci. \'Ens {\bf 12} (1979), 501--661.
\bibitem[IR]{IR} L. Illusie, M. Raynaud {\it Les suites spectrales associ\'ees au complexe de de Rham-Witt}, Publ. Math. IH\'ES {\bf 57} (1983), 73--212.
\bibitem[K1]{picfini} B. Kahn {\it Sur le groupe des classes d'un sch\'ema arithm\'etique (avec un appendice de Marc Hindry)}, Bull. SMF {\bf 134} (2006), 395--415.
\bibitem[K2]{ln} B. Kahn {\it D\'emonstration g\'eom\'etrique du th\'eor\`eme de Lang-N\'eron et 
formules de Shioda-Tate}, {\it in} Motives and algebraic cycles: a celebration in honour of Spencer J. Bloch, Fields Institute Communications {\bf 56}, AMS, 2009, 149--155.
\bibitem[K3]{adjoints} B. Kahn {\it Motifs et adjoints}, preprint, 2015, \url{http://arxiv.org/abs/1506.08386}.
\bibitem[KMP]{kmp} B. Kahn, J. P. Murre, C. Pedrini {\it On the transcendental part of the motive of a surface}, Algebraic cycles and motives (for J. P. Murre's 75th birthday), LMS Series {\bf 344} (2), Cambridge University Press, 2007, 143--202.
\bibitem[KP]{kp} B. Kahn, A. Pacheco {\it Fibrations et valeurs sp\'eciales de fonctions $L$}, in preparation.
\bibitem[KS]{birat-pure} B. Kahn, R. Sujatha {\it Birational motives, I: pure birational motives},  \url{http://arxiv.org/abs/0902.4902}, to appear in Annals of $K$-theory.
\bibitem[KT]{kato-trihan} K. Kato, F. Trihan {\it On the conjectures of Birch and Swinnerton-Dyer in characteristic $p>0$}, Invent. Math. {\bf 153} (2003), 537--592.
\bibitem[KM]{KM} N. Katz, W. Messing {\it Some consequences of Deligne's proof of the Riemann hypothesis for varieties over finite fields}, Invent. Math. {\bf 23} (1974), 73--77.
\bibitem[Kl]{kleiman} S. Kleiman {\it Motives}, {\it in} Algebraic geometry (Proc. Fifth Nordic Summer-School in Math., Oslo, 1970),  Wolters-Noordhoff, 1972, 53--82. 
\bibitem[M1]{milne2} J. S. Milne Etale cohomology, Princeton University Press, 1980.
\bibitem[M2]{milne} J. S. Milne {\it Values of zeta functions of varieties over finite fields}, Amer. J. Math. {\bf 108} (1988), 297--360. 
\bibitem[RZ]{rz} M. Rapoport, T. Zink, {\it \"Uber die lokale Zetafunktion von Shimuravarietten, Monodromiefiltration und verschwindende Zyklen in ungleicher Charakteristik}, Inv. Math.
{\bf 68} (1982), 21--201.
\bibitem[Ra]{raynaud} M. Raynaud {\it Caract\'eristique d'Euler-Poincar\'e d'un faisceau et
cohomologie des vari\'et\'es ab\'eliennes (d'apr\`es Ogg-Shafarevich et Grothendieck)}, S\'em.
Bourbaki 1964--66, exp. 286, 20 pp., {\it in} Dix expos\'es sur la cohomologie des sch\'emas, Masson-North Holland, 1968.
\bibitem[Saa]{saa} N. Saavedra Rivano Cat\'egories tannakiennes, Lect. Notes in Math. {\bf 265}, Springer, 1972.
\bibitem[Sa]{tsaito} T. Saito {\it Weight spectral sequences and independence of $l$}, J. Inst. Math. Jussieu {\bf 2} (2003), 583--634.
\bibitem[Scn]{schneider} P. Schneider {\it Zur Vermutung von Birch und Swinnerton-Dyer \"uber globalen Funktionenk\"orpern}, Math. Ann. {\bf 260} (1982), 495--510.
\bibitem[Sco]{scholl} A. Scholl {\it Classical motives}, {\it in} Motives (U. Jannsen, S. Kleiman, J.-P. Serre, eds.), Proc. Symp. Pure Math. {\bf 55} (I), 163--187.
\bibitem[S]{dpp} J.-P. Serre {\it Facteurs locaux de foncions z\^eta de vari\'et\'es alg\'ebriques (d\'efinitions et conjectures)}, S\'em. Delange-Pisot-Poitou 1969/70, exp. 19 = \OE uvres, Vol. II, {\bf 87}.
\bibitem[T1]{tate} J. Tate {\it On the conjectures of Birch and Swinnerton-Dyer and a geometric analog}, S\'em. Bourbaki 1965/66, exp. 306, 26 p., {\it in} Dix expos\'es sur la cohomologie des sch\'emas, Masson--North Holland, 1970, 189--214.
\bibitem[T2]{tatepoles} J. Tate {\it Algebraic cycles and poles of zeta-functions}, in Arithmetical
Algebraic Geometry, Harper and Row, New York, 1965.
\bibitem[W]{weil} A. Weil Vari\'et\'es ab\'eliennes et courbes alg\'ebriques, Hermann, 1948.
\bibitem[EGA3]{EGA3} A. Grothendieck \'El\'ements de g\'eom\'etriel alg\'ebrique (avec la collaboration de J. Dieudonn\'e), III: \'Etude cohomologique des faisceaux coh\'erents (1\`ere partie), Publ. Math. IH\'ES {\bf 11} (1961), 5--167.
\bibitem[SGA4]{SGA4} S\'eminaire de g\'eom\'etrie alg\'ebrique du Bois-Marie 1963--64 (SGA 4): Th\'eorie des topos et cohomologie \'etale des sch\'emas, Vol. III, Lect. Notes in Math. {\bf 305}, Springer, 1973.
\bibitem[SGA7]{SGA7} S\'eminaire de g\'eom\'etrie alg\'ebrique du Bois-Marie 1965--66 (SGA 7): Groupes de monodromie en g\'eom\'etrie alg\'ebrique, Vol. I, Lect. Notes in Math. {\bf 589}, Springer, 1977.
\end{thebibliography}
\end{document}